\documentclass[review,hidelinks,onefignum,onetabnum]{siamart220329}
\usepackage{lipsum}
\usepackage{amsfonts}
\usepackage{graphicx}
\usepackage{epstopdf}
\usepackage{algorithm,algorithmic}
\usepackage{graphicx}
\usepackage{epstopdf}
\usepackage{algorithmic}
\usepackage{verbatim}
\usepackage{multirow}
\usepackage[caption=false]{subfig}
\usepackage{amsmath,booktabs,ctable,threeparttable}
\usepackage{amssymb,boxedminipage}
\usepackage{microtype}
\allowdisplaybreaks[4]
\ifpdf
  \DeclareGraphicsExtensions{.eps,.pdf,.png,.jpg}
\else
  \DeclareGraphicsExtensions{.eps}
\fi

\newcommand{\mi}{\mathrm{i}}

\newsiamremark{rem}{Remark}
\newsiamremark{hypothesis}{Hypothesis}
\crefname{hypothesis}{Hypothesis}{Hypotheses}
\newsiamremark{claim}{Claim}
\newsiamremark{example}{Example}

\headers{A Power-like method of real skew-symmetric matrices}{Q. Zheng}

\title{A Power-like Method for Computing the Dominant Eigenpairs of
	Large Scale Real Skew-Symmetric Matrices\thanks{Submitted to the editors DATE.
\funding{The author was supported by the National Natural Science Foundation of China Grant (No. 62372467) and the Science Foundation of China University of Petroleum, Beijing (No. 2462021YJRC025, No. 2462024BJRC012).}}}

\author{Qingqing Zheng\thanks{College of Science, China University of Petroleum, Beijing 102249, China
  (\email{zhengqq@cup.edu.cn}).}
}

\usepackage{amsopn}

\begin{document}

\maketitle

\begin{abstract}
The power method is a basic method for computing the dominant eigenpair of a matrix.
In this paper, we propose a structure-preserving power-like method for computing the dominant conjugate pair of purely imaginary
eigenvalues and the corresponding eigenvectors of a large skew-symmetric matrix $S$, which works
on $S$ and its transpose $S^{T}$ alternately and is performed in {\em real} arithmetic.
We establish the rigorous and quantitative convergence of the proposed power-like method, and prove
that the approximations to the dominant eigenvalues converge twice as fast as those to the associated eigenvectors.
Moreover, we develop a deflation technique to compute several complex conjugate dominant eigenpairs of $S$. Numerical experiments show the effectiveness and efficiency of the new method.
\end{abstract}



\section{Introduction}
Consider the numerical solution of the following eigenproblem
\begin{equation}\label{pro}
	Sx=\lambda x, ~x\neq0,~\lambda\in\mathbb{C},
\end{equation}
where $S\in \mathbb{R}^{n\times n}$ is a real skew-symmetric matrix, i.e.,
$S^{T} = -S$, the superscript $T$ denotes the transpose of a matrix or vector, and
$(\lambda,x)$ is an eigenpair of $S$ with $\|x\|=1$ and $\|\cdot\|$ being the vector 2-norm
and the spectral norm of a matrix.
Problem (\ref{pro}) arises in a variety of applications, such as the solution of linear quadratic optimal control problems \cite{Mehrmann1991TheAL,Tsai2014}, matrix function computations \cite{Cardoso,Buono,BIT}, model reduction \cite{Mencik,Yan}, crack following in anisotropic materials \cite{Thomas}, polynomial eigenvalue problems (PEPs) \cite{Benner,Mackey,Mehrmann}, wave propagation solution for repetitive structures \cite{Mariana,Zhong}, and some others \cite{Greif,Mehrmann2012,Wimmer,PENKE2020102639,Tang}. Since the eigenvalues of a nonsingular real
skew-symmetric $S$ are all purely imaginary and come in conjugate pairs, $n$ must be even, so that
we can write the conjugate eigenvalues as $\lambda_{\pm}=\pm \mi \sigma$ and the corresponding
conjugate eigenvectors as $x_{\pm }$ in \eqref{pro}.

In contrast to the symmetric eigenvalue problem that has been studied intensively for several
decades \cite{Golub,Parlett,Stewart,Wilkinson}, the skew-symmetric eigenvalue problem has received very limited attention.
For small to medium sized skew-symmetric matrices, there are several numerical algorithms
available to compute the spectral decomposition of $S$ in real
arithmetic \cite{Paardekooper,PENKE2020102639,Ward}. For $S$ large scale,
the Matlab internal function {\sf eigs} \cite{Lehoucq} is based on the Arnoldi method
that generates a sequence of small upper Hessenberg matrices. In exact arithmetic, it is easily
justified that these small matrices reduce to the skew-symmetric form. However, in finite
precision, it is not the case, and the small matrices are of Hessenberg form
rather than skew-symmetric structure
due to round-off errors, causing their eigenvalues {\em never} to be purely imaginary and have
small but nonzero real parts and the computed eigenvectors lose their unique property to be described later. As a consequence, {\sf eigs} does not preserve the spectrum structure of $S$ and is not favorable.
Huang and Jia \cite{Huang} propose a skew-symmetric Lanczos bidiagonalization (SSLBD) method to compute
several extremal eigenpairs
based on the equivalence between the spectral decomposition of $S$  and a specific
structured singular value decomposition (SVD) of $S$, and develop an implicitly restarted SSLBD algorithm.
The algorithm is performed in real arithmetic and preserves the structures of
eigenvalues and eigenvectors of $S$, but it is quite involved and complicated.

The power method \cite{Golub,Saad,Stewart,Wilkinson} is a very old and simple
method for computing the dominant eigenpair of a general matrix $A$. Given
an initial vector $q_0$, this method generates the sequence of vectors $A^{k}q_{0}$ that can converge to
a dominant eigenvector in direction under the mild requirement that $q_{0}$ is not
deficient in the dominant eigenvector. Algorithm \ref{ClassicalPower} sketches the power method.
\begin{algorithm}
	\caption{The power method for a general matrix $A$}
	\label{ClassicalPower}
	\begin{enumerate}
		\item Choose a unit length initial vector $q_0$, and set $k=0$;
		\item {\bf While} not converged
		\item \quad\quad$q_{k+1}=Aq_{k}$;
		\item \quad\quad$\rho_{k}=q_{k}^{T}q_{k+1}$;
		\item \quad\quad$k=k+1$;
		\item \quad\quad$\mu_{k}=\dfrac{1}{\|q_{k}\|};$
		\item \quad\quad$q_{k}=\mu_{k}q_{k}$;
		\item {\bf End~while}.
	\end{enumerate}
\end{algorithm}
How to choose $q_0$ starting vector and design a convergence criterion is described in detail in \cite{Stewart}.
An advantage of the power method is that it requires only matrix-vector products.
However, the basic power method cannot converge for a real matrix with a pair of
complex conjugate dominant eigenvalues \cite{Wilkinson}. Therefore, it is not applicable to
the skew-symmetric matrix $S$. In order to fix this deficiency, Wilkinson \cite[p.579-81]{Wilkinson}
proposes a variant of the power method to compute complex dominant eigenpairs in real arithmetic, but he
neither gives necessary algorithmic details nor establishes quantitative convergence results
and designs a reliable stopping criterion.

To compute the dominant complex conjugate eigenpairs of $S$ in real arithmetic, we propose a
skew-symmetric power-like (SSP) method in this paper, which iterates with $S$ and $S^T=-S$ alternately.
We establish quantitative convergence results on the proposed method.
Based on it, we design a reliable and general-purpose stopping criterion.
In order to compute more than one complex conjugate dominant eigenpairs, we introduce a
structure-preserving deflation approach into the SSP method. Finally, we perform numerical experiments
on a number of test problems and illustrate the effectiveness of the SSP method with deflation.

The paper is organized as follows.
In Section~\ref{pre} we review some preliminaries. In Section~\ref{sec2} we propose
our SSP method, establish convergence results, and consider practical implementations.
In Section \ref{deflation}, we describe a structure-preserving deflation approach and apply
it to the SSP method.
Section \ref{secExam} reports numerical experiments. In Section \ref{con}, we conclude the paper.

\section{Preliminaries}\label{pre}

For a skew-symmetric matrix $S\in \mathbb{R}^{n\times n}$, there are the following properties,
which can be found in \cite{Huang}.

\begin{proposition}\label{propS} It holds that
	\begin{enumerate}
		\item $S$ is diagonalizable by a unitary similarity transformation;
		
		\item the eigenvalues $\lambda$ of $S$ are either purely imaginary or zero;
		
		\item the real and imaginary parts $u,v$ of the unit length
		eigenvectors $x_{\pm}=\frac{\sqrt{2}}{2}(u\pm\mi v)$ corresponding to
		purely imaginary eigenvalues $\pm\mi \sigma$ of $S$ are mutually orthogonal and
		have unit lengths: $\|u\|=\|v\|=1$.
		
	\end{enumerate}		
\end{proposition}

We can prove the following result, which plays a crucial role in proposing the later algorithms,
establishing their convergence and designing reliable stopping criteria.

\begin{theorem}\label{thm1}
	If $\frac{\sqrt{2}}{2}(u\pm\mi v)$ are the unit length eigenvectors of $S$ associated with the eigenvalues $\pm\mi\sigma$ and $w,z$ are an orthonormal basis
	of ${\rm span}\{u,v\}$, then $\frac{\sqrt{2}}{2}(w\pm\mi z)$ are also the
	unit-length eigenvectors of $S$ associated with $\pm\mi\sigma$.
\end{theorem}

\begin{proof}
Since $w,z$ and $u,v$ form two orthonormal bases of span$\{u,v\}$, respectively,
we must have
\begin{equation}\label{WZ1}
	(w,z)=(u,v)\begin{pmatrix}
		\cos\theta & -\sin\theta \\
		\sin\theta & \cos\theta\\
	\end{pmatrix}.
\end{equation}
or
\begin{equation}\label{WZ2}
	(w,z)=(u,v)\begin{pmatrix}
		\cos\theta & \sin\theta \\
		-\sin\theta & \cos\theta\\
	\end{pmatrix}.
\end{equation}
Suppose $\frac{\sqrt{2}}{2}(u+\mi v)$ is the eigenvector of $S$ corresponding to
the eigenvalue $\mi\sigma$. It is straightforward from $S(u+{\rm i}v)={\rm i}\sigma (u+{\rm i}v)$ that
$Su=-\sigma v, Sv=\sigma u$. If (\ref{WZ1}) holds, then $Sw=-\sigma z,Sz=\sigma w$, from which it follows that $S(w+\text{i}z)=-\sigma z+\text{i}\sigma w=\text{i}\sigma(w+\text{i}z)$, i.e.,
$(\text{i}\sigma, w+\text{i}z)$ is an eigenpair of $S$. If $w$ and $z$ satisfy (\ref{WZ2}), we obtain $S(w+\text{i}z)=-\text{i}\sigma(w+\text{i}z)$, i.e., $(-\text{i}\sigma,w+\text{i}z)$ is an eigenpair of $S$.
Since the eigenvectors of $S$ with the eigenvalues $\pm\mi\sigma$ are conjugate, $\frac{\sqrt{2}}{2}(w\pm\mi z)$ are also the
unit-length eigenvectors of $S$ associated with $\pm\mi\sigma$.
\end{proof}

\begin{rem}\label{rem1}
	Item 3 of Proposition~\ref{propS} and Theorem~\ref{thm1}
	show that we can write the eigenvectors $x_{\pm}=\frac{\sqrt{2}}{2}(w\pm \mi z)$
	of $S$ associated with the complex conjugate eigenvalues $\pm\mi\sigma$, provided that $w$ and $z$
	are an orthonormal basis of the eigenspace span$\{u,v\}$ of $S$
associated with  $\pm\mi\sigma$.

\end{rem}

\begin{rem}
Since $S$ with $n$ odd must have a zero eigenvalue, for brevity and ease of presentation
	we will always assume that $n$ is even in the sequel.
	Furthermore, we assume that $S$ is nonsingular so that all the eigenvalues are purely imaginary and thus come in pairs. A singular $S$ with $n$ even or odd has no effect on our algorithm to be proposed
	and the convergence results and analysis. We refer the reader to \cite{Huang} for more results and details
	on a singular $S$.
\end{rem}

For $n$ even, let $m=n/2$, and label the eigenvalues
$\lambda_{\pm j}=\pm\mi \sigma_j$ of $S$ as
\begin{equation}
	\sigma_1>\sigma_2\geq\cdots\geq \sigma_m.
\end{equation}
From $S(u_j\pm\mi v_j)=\pm\mi \sigma_j (u_j\pm\mi v_j)$
and $S^T=-S$, we obtain
\begin{equation}\label{AltEquation}
	Sv_{j}=\sigma_{j}u_{j},\ S^{T}u_{j}=\sigma_{j}v_{j} \ \ \mbox{and}\ \ \sigma_j=u_j^TSv_j,
\end{equation}
which means that $(\sigma_j,u_j,v_j)$ and $(\sigma_j,v_j,-u_j)$ are two singular triplets of $S$.
From the above, we have
\begin{equation}\label{eigSTS}
	S^{T}Su_{j}=\sigma_{j}^{2}u_{j}\ \ \mbox{and}\ \ S^{T}Sv_{j}=\sigma_{j}^{2}v_{j}.
\end{equation}
Particularly,  $\pm\text{i}\sigma_{1}$ are the
complex conjugate dominant eigenvalues of $S$, and the corresponding eigenvectors are
$x_{\pm 1}=\frac{\sqrt{2}}{2}(u_{1}\pm\text{i}v_{1})$. Then \eqref{eigSTS} indicates
that $\sigma_{1}^{2}$ is the dominant eigenvalue of $S^{T}S$ with
multiplicity two and $u_{1},v_{1}$ are the corresponding eigenvectors.

\section{A skew-symmetric power method by applying $S$ and $S^{T}$ alternately}\label{sec2}

Note that the complex conjugate eigenpairs
$(\text{i}\sigma_{1}, u_{1}+\text{i}v_{1})$ and $(-\text{i}\sigma_{1}, u_{1}-\text{i}v_{1})$ are
the dominant eigenpairs of $S$. Algorithm \ref{ClassicalPower} itself
cannot compute them. Next, we propose a variant of Algorithm \ref{ClassicalPower}
to carry out this task in real arithmetic.
Relation (\ref{AltEquation}) shows that the real and imaginary parts $u_{j}$ and $v_{j}$ of the eigenvector $x_j$
associated with the eigenvalue $\text{i}\sigma_{j}$ are parallel to $Sv_{j}$ and $S^{T}u_{j}$, respectively.
For the dominant eigenpairs $(\pm\mi\sigma_1,\frac{\sqrt{2}}{2}(u_1\pm\mi v_1)$, inspired by  (\ref{AltEquation}),
we can naturally design an iteration scheme: Given an initial unit length vector $q_0\in\mathbb{R}^n$,
construct unit length iterates $q_{k+1}=Sq_k/\|Sq_k\|$
and $q_{k+2}=S^Tq_{k+1}/\|S^Tq_{k+1}\|=-S q_{k+1}/\|Sq_{k+1}\|,\ k=0,1,\ldots$,
it is expected that $q_{k+1},q_{k+2}$ and $q_{k+1}^TSq_{k+2}$ converge to $u_1,v_1$
and $\sigma_1$, respectively. This is equivalent to applying
Algorithm \ref{ClassicalPower} to $S$ and $S^{T}$ alternately, as is described in Algorithm \ref{alg-powerSTS},
named the skew-symmetric power (SSP) method,
where each iteration uses two matrix-vector products with $S$.

\begin{algorithm}
	\caption{The SSP method by applying $S$ and $S^{T}$ alternately}
	\label{alg-powerSTS}
	\begin{enumerate}
		\item \quad Choose an unit length initial vector $q_0\in \mathbb{R}^n$, and set $k=0$;
		\item \quad {\bf While} not converged 
		\item \quad$q_{2k+1}=Sq_{2k}$; $q_{2k+1}=q_{2k+1}/\|q_{2k+1}\|$;
		\item \quad$q_{2k+2}=-Sq_{2k+1};$ $q_{2k+2}=q_{2k+2}/\|q_{2k+2}\|$;
		\item \quad$\rho_{k}=q_{2k+1}^{T}(Sq_{2k+2})$;
		
		\item \quad $k=k+1$
		\item \quad {\bf End while}
	\end{enumerate}
\end{algorithm}

We will prove the afore-anticipated convergence of Algorithm \ref{alg-powerSTS} and
establish a number of quantitative convergence results. Afterwards, we come back
to line 2 of Algorithm~\ref{alg-powerSTS} and design a general and reliable stopping criterion.
To this end, we need the following two lemmas.

\begin{lemma}\label{Orthpro} The sequences $\{q_{2k}\}_{k=1}^{\infty}$ and $\{q_{2k-1}\}_{k=1}^{\infty}$
	obtained by \Cref{alg-powerSTS} satisfy
	\begin{equation}\label{OrthOddEven}
		q_{2k}^{T}q_{2k-1}=0, k=1,2,\ldots,\infty,
	\end{equation}
	i.e., the two consecutive iterates of $\{q_{k}\}_{k=1}^{\infty}$ are orthogonal.
\end{lemma}

\begin{proof} From Algorithm \ref{alg-powerSTS} and $S^T=-S$, for any nonnegative integers $k$ and $\ell\geq 1$ we
inductively obtain
\begin{equation}\label{OddEvenk1}
	q_{k+2\ell}=c_{e}(S^{T}S)^{\ell}q_{k} \ \ \mbox{and}\ \ q_{k+2\ell-1}=c_{o}S(S^{T}S)^{\ell-1}q_{k},
\end{equation}
where $c_{e}$ and $c_{o}$ are two normalizing factors. Particularly, we have
\begin{equation*}
	q_{2\ell}=c_{e}(S^{T}S)^{\ell}q_0 \ \ \mbox{and}\ \  q_{2\ell-1}=c_{o}S(S^{T}S)^{\ell-1}q_0, \ell=1,2,\ldots,
\end{equation*}
i.e.,
\begin{equation}\label{OddEvenk}
	q_{2k}=c_{e}(S^{T}S)^{k}q_0\ \ \mbox{and}\ \ q_{2k-1}=c_{o}S(S^{T}S)^{k-1}q_0, k=1,2,\ldots.
\end{equation}
Therefore,
\begin{eqnarray*}
	q_{2k}^{T}q_{2k-1} &=&c_{e}c_{o}q_0^{T}(S^{T}S)^{k}S(S^{T}S)^{k-1}q_0, \\
	(q_{2k}^{T}q_{2k-1})^{T}&=& -c_{e}c_{o}q_0^{T}(S^{T}S)^{k}S(S^{T}S)^{k-1}q_0,
\end{eqnarray*}
meaning that $q_{2k}^{T}q_{2k-1}=0$ for $k=1,2,\ldots$.
\end{proof}

Now let us investigate the convergence of Algorithm \ref{ClassicalPower}. Assume that a matrix $A$ is diagonalizable
and its eigenvalues satisfy
\begin{equation}\label{Lam}
	\lambda_{1}=\lambda_{2}=\cdots=\lambda_{r}, |\lambda_{1}|>|\lambda_{r+1}|\geq\cdots\geq|\lambda_{n}|.
\end{equation}
The initial vector $q_0$ in Algorithm \ref{ClassicalPower} can be expressed as a linear
combination of the eigenvectors $x_i$ of $A$:
\begin{equation}\label{Xoexpress}
	q_0=\sum_{i=1}^{n}\beta_{i}x_{i}.
\end{equation}
So apart from a normalizing factor, $q_{k}$ in Algorithm \ref{ClassicalPower} is given by
\begin{equation}\label{Akx0} A^{k}q_0=\lambda_{1}^{k}\bigg(\sum_{i=1}^{r}\beta_{i}x_{i}+\sum_{i=r+1}^{n}\beta_{i}\left(\dfrac{\lambda_{i}}
	{\lambda_{1}}\right)^{k}x_{i}\bigg).
\end{equation}
Then provided that the initial vector $q_0$ is not deficient in the invariant subspace
${\rm span}\{x_1,\ldots,x_r\}$ of $A$ associated with the $r$ multiple dominant eigenvalues
$\lambda_{1}$, the sequence $q_k$ generated by Algorithm \ref{ClassicalPower}
converges to $\sum_{i=1}^{r}\beta_{i}x_{i}$
in direction, written as $x_*$ after normalization.
A convergence analysis on the power method can be found
in \cite{Saad,Stewart,Wilkinson}, and the convergence rate is $|\lambda_{r+1}/\lambda_1|$.
Furthermore, if $A$ is a normal matrix, these convergence results can be accurately quantified,
as shown below.

\begin{lemma}\label{converpower} Assume that $A$ is a normal matrix and its eigenvalues satisfy \eqref{Lam}.
	Then $q_{k}$ obtained by Algorithm \ref{ClassicalPower} satisfies
	\begin{equation}\label{Tan}
		|\tan\angle(q_{k},x_{*})|\leq\bigg|\frac{\lambda_{r+1}}{\lambda_{1}}\bigg|^{k}|\tan\angle(q_{0},x_{*})|,
	\end{equation}
	where $x_{*}\in\text{span}\{x_{1},x_{2},\ldots,x_{r}\}$ is the normalization of
	$\sum_{i=1}^{r}\beta_{i}x_{i}$ in \eqref{Akx0}, and
	\begin{equation}\label{Conrhok}
		\big|\rho_{k}-\lambda_{1}\big|\leq 2|\lambda_1|\tan^2\angle(q_{k},x_{*}).
	\end{equation}
\end{lemma}

\begin{proof} Denote $\epsilon_{0}=\sin\angle(q_{0},x_{*})$. From (\ref{Xoexpress}),
$q_{0}$ can be decomposed as the orthogonal direct sum
\begin{equation}\label{Rex0}
	q_{0}=\sqrt{1-\epsilon_{0}^{2}}x_{*}+\epsilon_{0}z_{0},
\end{equation}
where $x_{*}\in\text{span}\{x_{1},x_{2},\ldots,x_{r}\}$ and the unit length $z_{0}\in\text{span}\{x_{r+1},\ldots,x_{n}\}$ satisfies $x_{*}^{H}z_{0}=0$ with the superscript $H$ being the conjugate transpose of a vector.
Therefore,
\begin{equation*}
	A^{k}q_{0}=\sqrt{1-\epsilon_{0}^{2}}\lambda_{1}^{k}x_{*}+\epsilon_{0}A^{k}z_{0}.
\end{equation*}
Since the eigenvectors of a normal matrix are orthogonal, we have $x_{*}\bot A^{k}z_{0}$. Then
$$
	q_{k}= \dfrac{A^{k}q_{0}}{\|A^{k}q_{0}\|}=
	\sqrt{1-\epsilon_{0}^{2}} \dfrac{\lambda_{1}^{k}x_{*}}{\|A^{k}q_{0}\|}+\epsilon_{0}\dfrac{A^{k}z_{0}}{\|A^{k}q_{0}\|},
$$
which shows that
\begin{equation}\label{Tank}
	|\tan\angle(q_{k},x_{*})|=\dfrac{\epsilon_{0}\|A^{k}z_{0}\|}{\sqrt{1-\epsilon_{0}^{2}} \|\lambda_{1}^{k}x_{*}\|}
	\leq\bigg|\frac{\lambda_{r+1}}{\lambda_{1}}\bigg|^{k}|\tan\angle(q_{0},x_{*})|
\end{equation}
as $\|A^kz_0\|\leq |\lambda_{r+1}|^k$.
Let $\epsilon_{k}=\sin\angle(q_{k},x_{*})$. Then
\begin{equation*}
	q_{k}=\sqrt{1-\epsilon_{k}^{2}}x_{*}+\epsilon_{k}\hat{z}_{k},
\end{equation*}
where the unit length $\hat{z}_{k}\in\text{span}\{x_{r+1},\ldots,x_{n}\}$ satisfies $\hat{z}_{k}^{H}x_{*}=0$ and $\hat{z}_{k}^{H}Ax_{*}=0$.
So
\begin{eqnarray*}
	\rho_{k}&=& q_{k}^{H}A q_{k} \\
	&=&  \big(\sqrt{1-\epsilon_{k}^{2}}x_{*}+\epsilon_{k}\hat{z}_{k}\big)^{H}A  \big(\sqrt{1-\epsilon_{k}^{2}}x_{*}+\epsilon_{k}\hat{z}_{k}\big) \\
	&=& (1-\epsilon_{k}^{2})\lambda_{1}+ \epsilon_{k}^{2}\hat{z}_{k}^{H}A\hat{z}_{k}.
\end{eqnarray*}
Therefore,
$$
|\rho_k-\lambda_1|\leq (|\lambda_1|+|\lambda_{r+1}|)\epsilon_k^2\leq 2|\lambda_1|\epsilon_k^2,
$$
which proves \eqref{Conrhok} as $\epsilon_k\leq |\tan\angle(q_{k},x_{*})|$.
\end{proof}

\begin{rem}
	In case $\epsilon_{0}=1$ in \eqref{Rex0}, i.e., $|\tan\angle(q_{0},x_{*})|=\infty$,
	then the initial vector $q_0$ is deficient in $\text{span}\{x_1,\ldots,x_{r}\}$.
	In this case, the power method does not work. Otherwise,
	$q_k\rightarrow x_*$ with the convergence rate $|\lambda_{r+1}/\lambda_{1}|$,
and $\rho_k\rightarrow\lambda_1$ twice as fast as $q_k\rightarrow x_*$ as indicated by
\eqref{Conrhok}.
\end{rem}

The initial vector $q_0\in\mathbb{R}^n$ in Algorithm \ref{alg-powerSTS} can be expressed as
\begin{eqnarray}\label{x1Express}
	q_0&=&\sum_{i=1}^{m}[\alpha_{i}(u_{i}+\text{i}v_{i})+\bar{\alpha}_{i}(u_{i}-\text{i}v_{i})]\\
	&=&2(\alpha_{1,1}u_{1}-\alpha_{1,2}v_{1})+ 2\sum_{i=2}^{m}(\alpha_{i,1}u_{i}-\alpha_{i,2}v_{i})
\end{eqnarray}
where $\alpha_{i}=\alpha_{i,1}+\text{i}\alpha_{i,2}, i=1,\ldots,m$. Based on Lemmas \ref{Orthpro} and  \ref{converpower}, we can establish the following convergence results on Algorithm \ref{alg-powerSTS}.

\begin{theorem}\label{ConvergenceSTS}
	The sequences $\{q_{2k}\}_{k=1}^{\infty}$ and $\{q_{2k-1}\}_{k=1}^{\infty}$ obtained by \Cref{alg-powerSTS}
	satisfy
	\begin{equation}\label{sino}
		|\tan\angle(q_{2k},x_{e})|\leq\bigg(\dfrac{\sigma_{2}}{\sigma_{1}}\bigg)^{2k}|\tan\angle(q_{0},x_{e})|
	\end{equation}
	and
	\begin{equation}\label{sine}
		|\tan\angle(q_{2k-1},x_{o})|\leq \bigg(\dfrac{\sigma_{2}}{\sigma_{1}}\bigg)^{2(k-1)}|\tan\angle(Sq_0,x_{o})|,
	\end{equation}
	\begin{equation}\label{sine1}
		|\tan\angle(q_{2k-1},x_{o})|\leq \bigg(\dfrac{\sigma_{2}}{\sigma_{1}}\bigg)^{2k-1}|\tan\angle(q_0,x_{e})|,
	\end{equation}
	where the unit length vectors
	\begin{equation}\label{xoe}
		x_{o}=\dfrac{-1}{|\alpha_{1}|}(\alpha_{1,1}v_{1}+\alpha_{1,2}u_{1})\ \ \mbox{and}\ \
		x_{e}=\dfrac{1}{|\alpha_{1}|}(\alpha_{1,1}u_{1}-\alpha_{1,2}v_{1})
	\end{equation}
	are orthogonal and are the real and imaginary parts of $x_{\pm 1}$ of $S$ corresponding to the dominant
	eigenvalues $\pm\mi\sigma_1$.
	Moreover,
	\begin{equation}\label{rhok2}
		|\rho_{k}-\sigma_{1}|\leq \dfrac{1}{2}\tan^{2}\angle(q_{2k+1},x_{o})\sigma_{1}+O(\tan^{4}\angle(q_{2k+1},x_{o}))
		=O\left(\bigg(\dfrac{\sigma_{2}}{\sigma_{1}}\bigg)^{4k}\right).
	\end{equation}
\end{theorem}

\begin{proof} Note that $(S^{T}S)^{k}u_{i}=\sigma_{i}^{2k}u_{i}, (S^{T}S)^{k}v_{i}=\sigma_{i}^{2k}v_{i}$.
From (\ref{OddEvenk}) and (\ref{x1Express}), apart from the normalizing factor, $q_{2k}$ is given by
\begin{eqnarray}\label{X2kplus1}
	(S^{T}S)^{k}q_0&=& 2(S^{T}S)^{k}\sum_{i=1}^{m}(\alpha_{i,1}u_{i}-\alpha_{i,2}v_{i})\\ \nonumber
	&=& 2\sigma_{1}^{2k}\bigg[(\alpha_{1,1}u_{1}-\alpha_{1,2}v_{1})
	+\sum_{i=2}^{n}\big(\frac{\sigma_{i}}{\sigma_{1}}\big)^{2k}
	(\alpha_{i,1}u_{i}-\alpha_{i,2}v_{i})\bigg].
\end{eqnarray}
Therefore, as $k\rightarrow\infty$, the term $\alpha_{1,1}u_{1}-\alpha_{1,2}v_{1}$
in (\ref{X2kplus1}) dominates the right-hand side of the above relation, showing that
$q_{2k}$ converges to
$x_{e}$ defined by \eqref{xoe}. Relation (\ref{X2kplus1}) indicates that the sequence $\{q_{2k}\}_{k=1}^{\infty}$ is obtained
by applying Algorithm \ref{ClassicalPower} to the positive definite matrix
$S^{T}S$ with the initial vector $q_0$.
By Lemma \ref{converpower}, bound (\ref{sino}) holds.

Note that $(SS^{T})^{k-1}Su_{i}=-\sigma_{i}^{2k-1}v_{i}$ and $(SS^{T})^{k-1}Sv_{i}=\sigma_{i}^{2k-1}u_{i},\
i=1,2,\ldots,m$. Similarly, based on (\ref{OddEvenk}) and (\ref{x1Express}), apart from a
normalizing factor, $q_{2k-1}$ is given by
\begin{eqnarray}\label{X2k}
	&~&(SS^{T})^{k-1}S q_0\\ \nonumber
	&=& 2(S^{T}S)^{k-1}S\sum_{i=1}^{m}(\alpha_{i,1}u_{i}-\alpha_{i,2}v_{i})\\ \nonumber
	&=& 2\sigma_{1}^{2k-1}\bigg[(-\alpha_{1,1}v_{1}-\alpha_{1,2}u_{1})
	+\sum_{i=2}^{m}\big(\frac{\sigma_{i}}{\sigma_{1}}\big)^{2k-1}
	(-\alpha_{i,1}v_{i}-\alpha_{i,2}u_{i})\bigg].
\end{eqnarray}
Therefore, as $k\rightarrow\infty$, the term $-\alpha_{1,1}v_{1}-\alpha_{1,2}u_{1}$
in (\ref{X2k}) dominates, proving that $q_{2k-1}$ converges to the unit length vector $x_{o}$ defined by
\eqref{xoe}. Since the sequence $\{q_{2k-1}\}_{k=1}^{\infty}$ is obtained by applying Algorithm \ref{ClassicalPower}
to the positive definite matrix $SS^{T}$ with the initial vector $Sq_0$, by Lemma \ref{converpower}, we have
\begin{equation*}
	|\tan\angle(q_{2k-1},x_{o})|\leq|\tan\angle(Sq_0,x_{o})|\bigg(\dfrac{\sigma_{2}}{\sigma_{1}}\bigg)^{2(k-1)},
\end{equation*}
proving (\ref{sine}). It is straightforward from item 3 of Proposition~\ref{propS} to justify that
$x_o$ and $x_e$ are orthogonal. Therefore, they form an orthonormal basis of ${\rm span}\{u_1,v_1\}=
{\rm span}\{u_1+\mi v_1,u_1-\mi v_1\}$. By Theorem~\ref{thm1} and Remark~\ref{rem1},
the conjugate eigenvectors of $S$ associated with the eigenvalues $\pm\mi\sigma_1$ are
$\frac{\sqrt{2}}{2}(x_o\pm\mi x_e)$. Therefore, the sequences $q_{2k-1}$ and $q_{2k}$ obtained
by Algorithm~\ref{alg-powerSTS} converge to the real and imaginary parts of the complex
conjugate eigenvectors of $S$ corresponding to the dominant eigenvalues $\pm\mi\sigma_1$, respectively.

Let $\epsilon_{0}=\sin\angle(q_0,x_{e})$. From (\ref{x1Express}), $q_0$ can be decomposed into
the orthogonal direct sum
\begin{equation}\label{x1}
	q_0=\sqrt{1-\epsilon_{0}^{2}}x_{e}+\epsilon_{0} z_{o},
\end{equation}
where $x_{e}\in\text{span}\{u_{1}+\text{i}v_{1},u_{1}-\text{i}v_{1}\}$ and the unit length $z_{o}\in\text{span}\{u_{2}+\text{i}v_{2},u_{2}-\text{i}v_{2},\ldots,u_{m}+\text{i}v_{m},u_{m}-\text{i}v_{m}\}$ satisfy $Sx_{e}=\sigma_{1}x_{o}$, $x_{o}^{H}Sz_{o}=0$.
Then
\begin{equation*}
	Sq_0=\sqrt{1-\epsilon_{0}^{2}}Sx_{e}+\epsilon_{0} Sz_{o}=\sigma_{1}\sqrt{1-\epsilon_{0}^{2}}x_{o}+\epsilon_{0} Sz_{o},
\end{equation*}
from which it follows that
\begin{equation*}
	|\tan\angle(Sq_0,x_{o})|=\dfrac{\epsilon_{0}\|Sz_{o}\|}{\sigma_{1}\sqrt{1-\epsilon_{0}^{2}}\|x_{o}\|}
	\leq\dfrac{\epsilon_{0}}{\sqrt{1-\epsilon_{0}^{2}}}\frac{\sigma_{2}}{\sigma_{1}}=
	|\tan\angle(q_0,x_{e})|\frac{\sigma_{2}}{\sigma_{1}}.
\end{equation*}
As a result, by (\ref{sine}) we have proved (\ref{sine1}). 	

Now we prove the convergence of $\{\rho_{k}\}_{k=1}^{\infty}$. Let $\epsilon_{o,k}=\sin\angle(q_{2k+1},x_{o})$. From (\ref{X2k}), $q_{2k+1}$ can be written as
\begin{equation}\label{x1o}
	q_{2k+1}=\sqrt{1-\epsilon_{o,k}^{2}}x_{o}+\epsilon_{o,k} z_{o,k},
\end{equation}
where $x_{o}\in\text{span}\{u_{1}+\text{i}v_{1},u_{1}-\text{i}v_{1}\}$ is the vector
defined as in (\ref{xoe}), and the unit length vector
$z_{o,k}\in\text{span}\{u_{2}+\text{i}v_{2},u_{2}-\text{i}v_{2},\ldots,u_{m}+\text{i}v_{m},u_{m}-\text{i}v_{m}\}$
satisfies $x_{o}^{T}z_{o,k}=0$.
Since $x_{e}^{T}(Sz_{o,k})=0$
and $Sx_{o}=-\sigma_{1}x_{e}$, from (\ref{x1o}), we have
\begin{eqnarray*}
	\|Sq_{2k+1}\|^{2}&=&(\sqrt{1-\epsilon_{o,k}^{2}}Sx_{o}+\epsilon_{o,k} Sz_{o,k})^{T}(\sqrt{1-\epsilon_{o,k}^{2}}Sx_{o}+\epsilon_{o,k} Sz_{o,k}) \\ \nonumber
	&=& (-\sqrt{1-\epsilon_{o,k}^{2}}\sigma_{1}x_{e}+\epsilon_{o,k} Sz_{o,k})^{T}(-\sqrt{1-\epsilon_{o,k}^{2}}\sigma_{1}x_{e}+\epsilon_{o,k} Sz_{o,k}) \\ \nonumber
	&=& (1-\epsilon_{o,k}^{2})\sigma_{1}^{2}+\epsilon_{o,k}^{2}z_{o,k}^{T}S^{T}Sz_{o,k}.
\end{eqnarray*}
Then
\begin{eqnarray*}
	\rho_{k}&=& q_{2k+1}^{T}S q_{2k+2} \\ \nonumber
	&=&  \dfrac{-q_{2k+1}^{T}S^{2}q_{2k+1}}{\|-Sq_{2k+1}\|} \\ \nonumber
	&=& \|Sq_{2k+1}\| \\ \nonumber
&=&\sqrt{(1-\epsilon_{o,k}^{2})\sigma_{1}^{2}+\epsilon_{o,k}^{2}z_{o,k}^{T}S^{T}Sz_{o,k}}\\ \nonumber
&=&\sigma_{1}\sqrt{1-\epsilon_{o,k}^{2}\bigg(1-\dfrac{z_{o,k}^{T}S^{T}Sz_{o,k}}{\sigma_{1}^{2}}\bigg)}\\ \nonumber
 &=&\sigma_{1}\bigg[1-\dfrac{1}{2}\epsilon_{o,k}^{2}\bigg(1-\dfrac{z_{o,k}^{T}S^{T}Sz_{o,k}}{\sigma_{1}^{2}}\bigg)+O(\epsilon_{o,k}^{4})\bigg].
\end{eqnarray*}
Since  $|z_{o,k}^{T}S^{T}Sz_{o,k}|\leq\sigma_{2}^{2}<\sigma_{1}^{2}$, from the above we obtain
\begin{equation}\label{Boundrho1}
|\rho_{k}-\sigma_{1}| \leq \frac{1}{2}\epsilon_{o,k}^{2}\sigma_{1}
+O\big(\epsilon_{o,k}^{4}\big),
\end{equation}
which, together with \eqref{sine1}, proves \eqref{rhok2}.
\end{proof}

\begin{rem}\label{tanoe}
	This theorem indicates that $\rho_k$ converges twice as fast as $q_{2k}$ and $q_{2k-1}$ and
	the error of $\rho_k$ is approximately squares of the latter two ones.
\end{rem}

\begin{rem}\label{strc1}
By Algorithm \ref{alg-powerSTS} and Theorem~\ref{ConvergenceSTS}, at iteration $k$
we can construct complex conjugate approximations
$\big(\pm\mi \rho_{k},\frac{\sqrt{2}}{2}(q_{2k+1}\pm\mi q_{2k+2})\big)$ to
the complex conjugate dominant eigenpairs
$\big(\pm\mi \sigma_1,\frac{\sqrt{2}}{2}(u_1\pm\mi v_1)\big)$ of $S$.
This forms a complete algorithm when a reliable
stopping criterion is provided.
From Lemma \ref{Orthpro} and Theorem \ref{ConvergenceSTS},
the resulting Algorithm \ref{alg-powerSTS} is structure-preserving:
The approximate eigenvalues $\pm\mi \rho_{k}$ obtained are purely imaginary, and
the real and imaginary parts of the approximate eigenvectors $\frac{\sqrt{2}}{2}(q_{2k+1}\pm\mi q_{2k+2})$ are equal in
length and mutually orthogonal.
\end{rem}

Let us come back to line 2 of Algorithm \ref{alg-powerSTS} and design a reliable and
general-purpose stopping criterion. By definition, the absolute residual norms of approximate
eigenpairs $\big(\pm\mi \rho_{k},\frac{\sqrt{2}}{2}(q_{2k+1}\pm\mi
q_{2k+2})\big)$ are
\begin{eqnarray}\label{Aberror}
	\|r_{\pm k}\|&=&\bigg\|\frac{\sqrt{2}}{2}S(q_{2k+1}\pm\mi q_{2k+2})\mp\mi\frac{\sqrt{2}}{2} \rho_{k}(q_{2k+1}\pm\mi q_{2k+2})\bigg\| \\ \nonumber
	&=&\frac{\sqrt{2}}{2} \bigg\|(Sq_{2k+1}+ \rho_{k}q_{2k+2})\mp\mi(\rho_{k}q_{2k+1}- Sq_{2k+2} )\bigg\| \\ \nonumber
	&=&\frac{\sqrt{2}}{2} \sqrt{\|Sq_{2k+1}+ \rho_{k}q_{2k+2}\|^{2}+\|Sq_{2k+2}-\rho_{k}q_{2k+1}\|^{2}}.
\end{eqnarray}
Therefore, in practical computations, Algorithm \ref{alg-powerSTS} stops
when the relative residual norms
\begin{equation}\label{stop1}
	\text{ERR}_{k}=\frac{\|r_{\pm k}\|}{|\rho_k|}< tol
\end{equation}
for a prescribed convergence tolerance $tol$. This is a general and reliable
stopping criterion for $tol\geq O(\epsilon_{\rm mach})$ with $\epsilon_{\rm mach}$
being the machine precision and the constant in the big $O(\cdot)$ being generic, say
100.

As a matter of fact, \Cref{alg-powerSTS} is more general and does not require that
the complex conjugate eigenvalues $\pm\mi\sigma_1$ be simple. To see this, we
show that Theorem \ref{ConvergenceSTS} can be extended to the case that the dominant eigenvalues $\pm\mi\sigma_{1}$
are multiple. Assume that their multiplicities are $r$ and $x_{\pm i}=\frac{\sqrt{2}}{2}(u_i\pm\mi v_i)$
are the associated $r$ orthonormal eigenvectors. Then (\ref{X2kplus1}) can be rewritten as
\begin{eqnarray}\label{X2kplus1M}
	(S^{T}S)^{k}q_0&=& 2(S^{T}S)^{k}\sum_{i=1}^{m}(\alpha_{i,1}u_{i}-\alpha_{i,2}v_{i})\\ \nonumber
	&=& 2\sigma_{1}^{2k}\bigg[\sum_{i=1}^{r}(\alpha_{i,1}u_{i}-\alpha_{i,2}v_{i})
	+\sum_{i=r+1}^{m}\big(\frac{\sigma_{i}}{\sigma_{1}}\big)^{2k}
	(\alpha_{i,1}u_{i}-\alpha_{i,2}v_{i})\bigg].
\end{eqnarray}
Therefore, as $k\rightarrow\infty$, $q_{2k}$ converges to the following unit length vector
\begin{equation}\label{xeM}
x_{e}=\dfrac{1}{n_{\alpha}}\sum_{i=1}^{r}(\alpha_{i,1}u_{i}-\alpha_{i,2}v_{i}),
\end{equation}
where $n_{\alpha}=\sum_{i=1}^{r}(\alpha^{2}_{i,1}+\alpha^{2}_{i,2})$.
Meanwhile, (\ref{X2k}) can be rewritten as
\begin{eqnarray}\label{X2kM}
	&~&(SS^{T})^{k-1}S q_0\\ \nonumber
	&=& 2\sigma_{1}^{2k-1}\bigg[\sum_{i=1}^{r}(-\alpha_{i,1}v_{i}-\alpha_{i,2}u_{i})
	+\sum_{i=r+1}^{m}\big(\frac{\sigma_{i}}{\sigma_{1}}\big)^{2k-1}
	(-\alpha_{i,1}v_{i}-\alpha_{i,2}u_{i})\bigg],
\end{eqnarray}
proving that $q_{2k-1}$ converges to the following unit length vector
\begin{equation}\label{xoM}
	x_{o}=\dfrac{1}{n_{\alpha}}\sum_{i=1}^{r}(-\alpha_{i,1}v_{i}-\alpha_{i,2}u_{i}).
\end{equation}

It is straightforward to justify that $x_e$ and $x_o$ in (\ref{xeM}) and (\ref{xoM}) are orthogonal.
Moreover, we have
\begin{equation*}
	Sx_{o}=-\sigma_{1}x_{e} \ \mbox{and}\ \ Sx_{e}=\sigma_{1}x_{o},
\end{equation*}
from which it follows that
$$
S(x_{o}\pm\mi x_{e})=\pm\mi\sigma_{1}(x_{o}\pm\mi x_{e}),
$$
meaning that $\frac{\sqrt{2}}{2}(x_o\pm \mi x_e)$ are the unit-length eigenvectors of $S$ corresponding to the complex
conjugate dominant eigenvalues $\pm \mi \sigma_1$.
As we have proved above, the sequences $q_{2k-1}$ and $q_{2k}$ obtained by Algorithm \ref{alg-powerSTS} converge to the real and imaginary parts $x_o$ and $x_e$.
Moreover, (\ref{sino})--(\ref{sine1}) hold with the vectors $x_{e}$ and $x_{o}$ in (\ref{xoe}) replaced by the vectors in (\ref{xeM}) and (\ref{xoM}), respectively.

The above shows that \Cref{alg-powerSTS}  directly
works on $S$ whose dominant eigenvalues are multiple.

\section{The SSP method with deflation for computing several complex conjugate dominant eigenpairs}\label{deflation}

\Cref{alg-powerSTS} can compute only the complex conjugate eigenpairs $(\pm\mi\sigma_1,x_{\pm 1})$.
Suppose the $s$ complex conjugate eigenpairs $(\pm\mi\sigma_i,x_{\pm i})$ with $s>1$ are of interest 
and the eigenvalues are labeled as 
$$
\sigma_1\geq\sigma_2\geq\cdots\geq \sigma_s>\sigma_{s+1}\geq\cdots\geq\sigma_m.
$$
Then
\Cref{alg-powerSTS} itself is not applicable.  To this end, we introduce an effective deflation technique into \Cref{alg-powerSTS} to carry out this task.

Recall from Theorem 2.1 of \cite{Huang} that
\begin{equation}\label{SVD}
	S=\begin{pmatrix}
		U_m & V_m  \\
	\end{pmatrix}\begin{pmatrix}
	\Sigma_m &   \\
	 &  \Sigma_m \\
\end{pmatrix}\begin{pmatrix}
V_m&-U_m  \\
\end{pmatrix}^{T}\\,
\end{equation}
where $U_m=(u_{1},u_{2},\ldots,u_{m})\in \mathbb{R}^{n\times m}$ and $V_m=(v_{1},v_{2},\ldots,v_{m})\in \mathbb{R}^{n\times m}$ are orthonormal and biorthogonal, and $\Sigma_m=\text{diag}\{\sigma_{1},\sigma_{2},\ldots,\sigma_{m}\}\in\mathbb{R}^{m\times m}$ with
$\sigma_{i}>0,i=1,2,\ldots,m$. Suppose the $i<s$ complex conjugate dominant eigenpairs
$(\pm\mi\sigma_{i},x_{\pm i})$ have
been available and we want to compute the next eigenpairs $(\pm\mi\sigma_{i+1},x_{\pm (i+1)})$. Then we can deflate those available ones from the SVD \eqref{SVD} of $S$ in the following way:
\begin{equation}\label{newS}
S_{i} = S-\begin{pmatrix}
	U_{i} & V_{i}  \\
\end{pmatrix}\begin{pmatrix}
	\Sigma_{i} &   \\
	&  \Sigma_{i} \\
\end{pmatrix}\begin{pmatrix}
	V_{i}&-U_{i}  \\
\end{pmatrix}^{T}\\,
\end{equation}
where
\begin{equation*}
U_{i}=(u_{1},u_{2},\ldots,u_{i})\in \mathbb{R}^{n\times i}, \ \
V_{i}=(v_{1},v_{2},\ldots,v_{i})\in \mathbb{R}^{n\times i}
\end{equation*}
and
 $\Sigma_{i}=\text{diag}\{\sigma_{1},\sigma_{2},\ldots,\sigma_{i}\}\in\mathbb{R}^{i\times i}$.
It is easy to justify that $S_i$ is skew-symmetric and its eigenpairs are $(0,x_{\pm j}),\ j=1,2,\ldots,i$
and $(\pm\mi\sigma_j,x_{\pm j}),\ j=i+1,\ldots,m$. Therefore, $(\pm\mi\sigma_{i+1},x_{\pm (i+1)})$
are the complex conjugate dominant eigenpairs of $S_i, i=1,2,\ldots,s-1$, so that theoretically
we can apply \Cref{alg-powerSTS} to $S_i$ and computes $(\pm\mi\sigma_{i+1},x_{\pm (i+1)})$. Proceed
in such a way until all the desired $(\pm\mi\sigma_i,x_{\pm i}),\ i=1,2,\ldots,s$ are found.

Now we discuss how to apply \Cref{alg-powerSTS} in practical computations to compute the $s$ 
complex conjugate dominant eigenpairs. Note that the eigenpairs $(\pm\mi\sigma_j,x_{\pm j}),\ 
j=1,2,\ldots,i$ cannot be computed exactly. Instead, for a given stopping tolerance, 
what we obtain are their {\em converged} approximations 
$(\pm\mi\tilde{\sigma}_j,\tilde{x}_{\pm j})$ with $\tilde{x}_{\pm 
j}=\tilde{u}_{j}\pm\mi\tilde{v}_{j}, \ j=1,2,\ldots,i$. In computations, the matrix 
$S_{i}$ in \eqref{newS} is replaced by the following skew-symmetric matrix
\begin{equation}\label{newStidle}
	\tilde{S}_{i} = S-\begin{pmatrix}
		\tilde{U}_{i} & \tilde{V}_{i}  \\
	\end{pmatrix}\begin{pmatrix}
		\tilde{\Sigma}_{i} &   \\
		&  \tilde{\Sigma}_{i} \\
	\end{pmatrix}\begin{pmatrix}
		\tilde{V}_{i}&-\tilde{U}_{i}  \\
	\end{pmatrix}^{T}\\,
\end{equation}
where
\begin{equation*}
	\tilde{U}_{i}=(\tilde{u}_{1},\tilde{u}_{2},\ldots,\tilde{u}_{i})\in \mathbb{R}^{n\times i}, \ \
	\tilde{V}_{i}=(\tilde{v}_{1},\tilde{v}_{2},\ldots,\tilde{v}_{i})\in \mathbb{R}^{n\times i}
\end{equation*}
and
$\tilde{\Sigma}_{i}=\text{diag}\{\tilde{\sigma}_{1},\tilde{\sigma}_{2},\ldots,\tilde{\sigma}_{i}\}\in\mathbb{R}^{i\times i}$.

Since $\tilde{S}_i$ in \eqref{newStidle} is skew-symmetric, \Cref{alg-powerSTS} can be applied to $\tilde{S}_i$ to compute the converged approximations $(\pm\mi\tilde{\sigma}_{i+1},\tilde{u}_{i+1}\pm\tilde{v}_{i+1})$ to the 
complex conjugate dominant eigenpairs $(\pm\mi\sigma_{i+1},u_{i+1}\pm v_{i+1})$. 
However, it is particularly important that we do not need to form $\tilde{S}_i$ explicitly since
the only action of $\tilde{S}_i$ in \Cref{alg-powerSTS}  is to form two matrix-vector products, 
as shown in step 5 of Algorithm \ref{alg-MpowerSTS}. When $s=1$, \Cref{alg-MpowerSTS} reduces to \Cref{alg-powerSTS}. 

\begin{algorithm}
	\caption{Compute $s$ complex conjugate dominant eigenpairs by \Cref{alg-powerSTS}}
	\label{alg-MpowerSTS}
	\begin{enumerate}
		\item \quad  {\bf If} $s=1$, apply \Cref{alg-powerSTS} to $S$; {\bf Otherwise} go to Step 2;
		\item \quad {\bf For} $i=2:s$
		\item \quad Choose an unit length initial vector $q_0\in \mathbb{R}^n$, and set $k=0$;
		\item \quad {\bf While} not converged 
		\item \quad\quad
		\begin{eqnarray*}
		\quad	&&q_{2k+1}=Sq_{2k}-\sum_{j=1}^{i}(\tilde{\sigma}_{j}\tilde{v}_{j}^{T}q_{2k}\tilde{u}_{j}-\tilde{\sigma}_{j}\tilde{u}_{j}^{T}q_{2k}\tilde{v}_{j}); q_{2k+1}=q_{2k+1}/\|q_{2k+1}\| \\
			&&q_{2k+2}=-Sq_{2k+1}+\sum_{j=1}^{i}(\tilde{\sigma}_{j}\tilde{v}_{j}^{T}q_{2k+1}\tilde{u}_{j}-\tilde{\sigma}_{j}\tilde{u}_{j}^{T}q_{2k+1}\tilde{v}_{j}); q_{2k+2}=q_{2k+2}/\|q_{2k+2}\|
		\end{eqnarray*}
	
		\item \quad$\rho_{k}=q_{2k+1}^{T}(Sq_{2k+2})$;
		
		\item \quad $k=k+1$
		\item \quad {\bf End while}
		\item \quad $\tilde{u}_{i+1}=q_{2k+1}$; $\tilde{v}_{i+1}=q_{2k+2}$; $\tilde{\sigma}_{i+1}=\rho_{k}$;
		\item \quad {\bf End for}
	\end{enumerate}
\end{algorithm}

We must be careful on the stopping criterion in \Cref{alg-MpowerSTS} since (i) we are 
concerned with the eigenvalue problem of $S$ rather than that of $\tilde{S}_i$ and (ii) we must take 
numerical backward stability into account. To this end, 
when computing the $i$-th complex conjugate eigenpairs of $S$ with $i\geq2$, we take the stopping 
criterion as 
\begin{equation}\label{stop2}
	\text{ERR}_{k}=\frac{\|r_{\pm k}\|}{\tilde{\sigma}_{1}}< tol
\end{equation}
with $tol\geq O(\epsilon_{\rm mach})$ being a prescribed convergence tolerance, where $\|r_{\pm k}\|$ are
the absolute residual norms of $\big(\pm\mi \rho_{k},\frac{\sqrt{2}}{2}(q_{2k+1}\pm\mi
q_{2k+2})\big)$ as the approximate 
complex conjugate eigenpairs of $S$ rather than $\tilde{S}_{i}$ (cf. \eqref{Aberror}). Keep in mind 
that the denominator in \eqref{stop2} is always the converged $\tilde{\sigma}_1\approx \|S\|$ 
for $i\geq 2$. 

\section{Numerical examples}\label{secExam}
In this section, we report numerical experiments to show the performance of \Cref{alg-powerSTS} and
 \Cref{alg-MpowerSTS} for computing the complex conjugate dominant eigenpairs of a large real skew-symmetric $S$. All the
 experiments were performed on a PC with a 2.60 GHz central processing unit
 Intel(R) Core(TM) i9-11950HM, 32.00 GB memory under Microsoft Windows 10 64-bit operating system using MATLAB R2021b with $\epsilon_{\rm mach}=2.22\times 10^{-16}$. The initial vector in the algorithms
 was taken as the unit-length vector normalized from $S(1,1,\ldots,1)^{T}$ and the stopping tolerance of the relative residual norms $\text{ERR}_{k}$ defined in (\ref{stop1}) and \eqref{stop2} were fixed as $tol=10^{-8}$.
 In the follow-up numerical results, the notations IT, TIME and MV
 denote the number of iterations, CPU time in seconds and matrix-vector products with $S$. 

\begin{example}\label{exa1}
	This example is from Section $11$ of \cite{Greif}, see also \cite{Huang}. The skew-symmetric matrix $S$ 
is obtained by the approximate finite difference discretization of a constant convective term on the unit cube $(0,1)\times(0,1)\times(0,1)$ in three dimensions:
	\begin{equation}
		S=I_{l}\otimes I_{l}\otimes T_{l}(\zeta_{1})+I_{l}\otimes T_{l}(\zeta_{2})\otimes I_{l}+T_{l}(\zeta_{3})\otimes I_{l}\otimes I_{l},
	\end{equation}
	where $l$ is a given positive integer and
	\begin{equation*}
		T_{l}(\zeta_{k})=\begin{pmatrix}
			0 & \zeta_{k} &  & \\
			-\zeta_{k} &\ddots & \ddots & \\
			& \ddots & \ddots & \zeta_{k}\\
			&  & -\zeta_{k} &0 \\
		\end{pmatrix}\in \mathbb{R}^{l\times l}, \ \ k=1,2,3
	\end{equation*}
	are three skew-symmetric tridiagonal Toeplitz matrices with $\zeta_{k}, k=1,2,3$ being real scalars. 
We take $\zeta_{1},\zeta_{2},\zeta_{3}$ to be $0.4,0.5,0.6$, and  
compute the one and five complex conjugate dominant eigenpairs of the three $S$'s resulting from   
$l=8,16,32$, which correspond to $n=512, \ 4096,\ 32768$. \Cref{tab:ratio} lists the convergence rates $\gamma_{i}=\sigma_{i+1}/\sigma_{i}, \ i=1,\ldots,5$ for the five complex conjugate dominant eigenpairs of $S$ when $l=8,16,32$, respectively.  \Cref{tab:comp1} reports
the results obtained by \Cref{alg-powerSTS} and \Cref{alg-MpowerSTS}. Moreover, \Cref{RES1} shows the relative residual norms of \Cref{alg-MpowerSTS} and the average CPU time (aveCPU) of each iteration of \Cref{alg-MpowerSTS} with $l=32$ and $s=1,\ldots,5$. 
\end{example}

\begin{table}[h]
	\caption{The convergence rate of \Cref{alg-MpowerSTS} for computing 
the complex conjugate dominant eigenpairs of $S$ in Example \ref{exa1}.
		\label{tab:ratio}}
	\begin{center}
		\def\arraystretch{1.5}
		\small
		\centering\tabcolsep12.0pt
		\begin{tabular}{c|ccccc}
			\hline
			$l$	&  $\gamma_{1}$ & $\gamma_{2}$ &$\gamma_{3}$ & $\gamma_{4}$ & $\gamma_{5}$ \tabularnewline	\hline
			8  & 0.9507& 0.9870& 0.9869&  0.9601& 0.9861 \tabularnewline
			16 & 0.9863& 0.9965& 0.9965&  0.9895& 0.9965\tabularnewline
			32  & 0.9964& 0.9991& 0.9991&0.9973& 0.9991 \tabularnewline		
			\hline
		\end{tabular}
	\end{center}
\end{table}

\begin{table}[h]
	\caption{The IT, TIME, MV's used by \Cref{alg-powerSTS} and \Cref{alg-MpowerSTS} to compute the complex conjugate dominant eigenpairs of $S$ in Example \ref{exa1}.
		\label{tab:comp1}}
	\begin{center}
		\def\arraystretch{1.5}
		\small
		\centering\tabcolsep12.0pt
		\begin{tabular}{c|ccc|ccc}
			\hline
			$l$ &
			\multicolumn{3}{c|}{$s=1$, \Cref{alg-powerSTS}}&\multicolumn{3}{c}{$s=5$, \Cref{alg-MpowerSTS}}\tabularnewline
			&  IT & TIME &MV &  IT & TIME &MV \tabularnewline	\hline
			8  & 164& 0.0017& 329&  1975& 0.0438& 3955 \tabularnewline
			16 & 551& 0.0269& 1103&  6865& 0.9529& 13735 \tabularnewline
			32  & 1906& 0.8695& 3813&23720& 50.0593& 47445\tabularnewline		
			\hline
		\end{tabular}
	\end{center}
\end{table}

\Cref{tab:ratio} indicates that each of $\gamma_i,i=1,2,\ldots,5$ increases as $l$ does. Consequently, 
by \Cref{converpower}, it is expected that the number of iterations increases $l$ does for computing 
each of the five complex conjugate dominant eigenpairs of $S$. This is indeed confirmed by the results 
in \Cref{tab:comp1} for $s=1$. For the other cases, the results are similar but we do not report them due 
to space. 

In the meantime, from \Cref{tab:ratio}, we expect that fewer number of iterations should be used to 
compute the first and fourth complex conjugate dominant eigenpairs of $S$ than the others because the 
convergence rates $\gamma_{1}$ and $\gamma_{4}$ are smaller than $\gamma_{2},\gamma_{3}$ and $\gamma_{5}$ 
for three different $l$'s. \Cref{RES1}(a) clearly confirms our theoretical predictions, where the five 
curves depict the convergence processes of computing the five complex conjugate dominant eigenpairs for 
$l=32$. From \Cref{RES1}(b), we can also see that the average CPU time per iteration monotonically 
increases as $s$ increases from $1$ to $5$. This is because the deflation process in \Cref{alg-MpowerSTS} 
needs to implement two matrix-vector products associated with $S$ in step 5, which requires more 
computational costs when $s$ becomes large. 
\begin{figure}[h!]
	\centering\tabcolsep=-8mm
	{\footnotesize{\doublerulesep1pt
			\begin{tabular}
				[c]{c}%
				\includegraphics[height=4.0in]{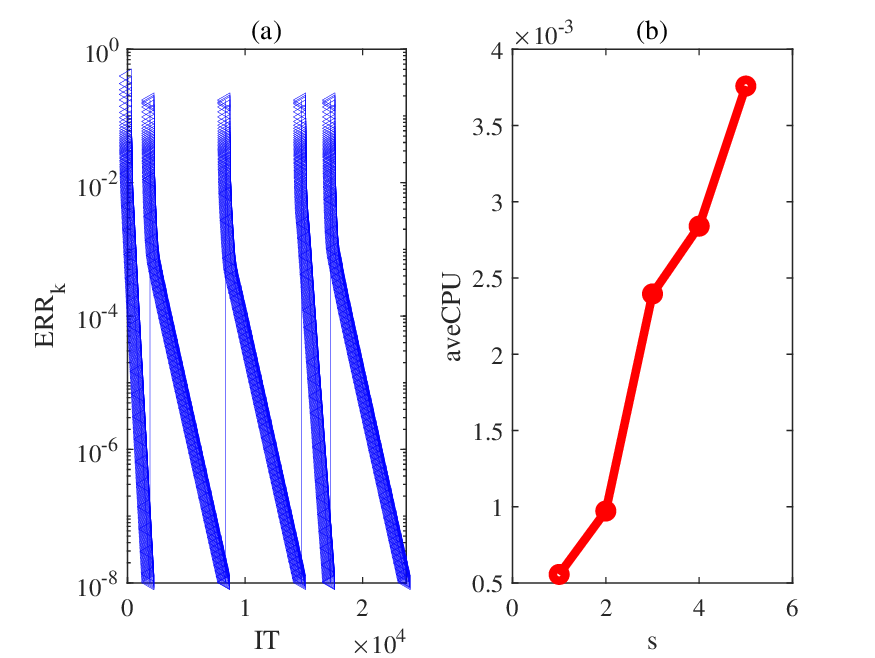}
	\end{tabular}}}
	\caption{(a): convergence curves of the relative residual norms of \Cref{alg-MpowerSTS} for $l=32$, 
$s=5$; (b): the aveCPU for $s=1,2,\ldots,5$.}
	\label{RES1}
\end{figure}

\begin{figure}[h!]
	\centering\tabcolsep=-8mm
	{\footnotesize{\doublerulesep1pt
			\begin{tabular}
				[c]{c}%
				\includegraphics[height=4.0in]{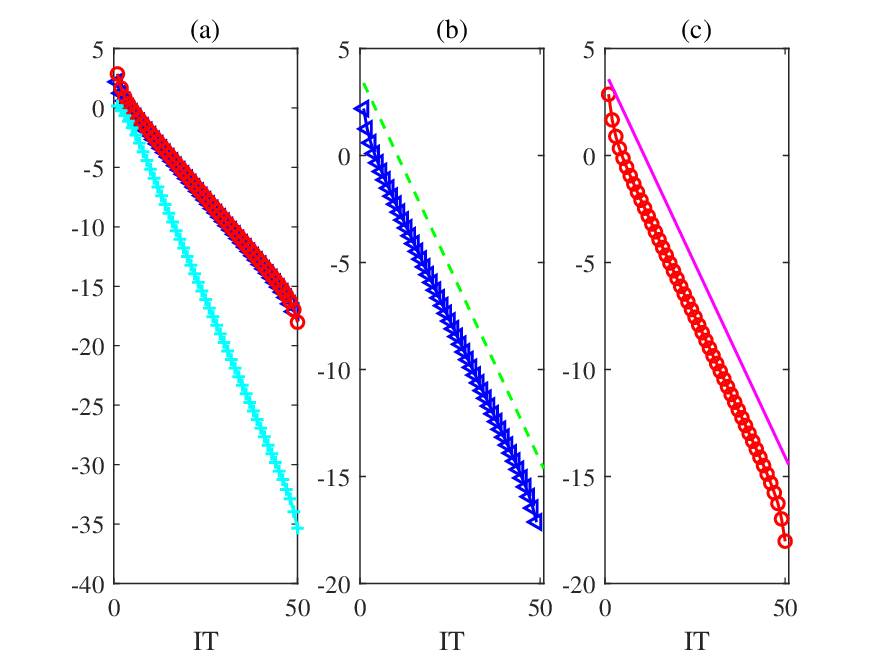}
	\end{tabular}}}
	\caption{Curves of $\log\big(|\tan\angle(x_{2k-1},x_{o})|\big)${\rm (}circle{\rm )},  $\log\big(|\tan\angle(x_{2k},x_{e})|\big)${\rm (}triangle{\rm )}, $\log\big(|\rho_{k}-\sigma_{1}|\big)$ {\rm (}plus{\rm )},  $\log\big(\big(\dfrac{\sigma_{2}}{\sigma_{1}}\big)^{2k}|\tan\angle(q_{0},x_{e})|\big)$ {\rm (}dashed line{\rm )} and $\log\big(\big(\dfrac{\sigma_{2}}{\sigma_{1}}\big)^{2k-1}|\tan\angle(q_{0},x_{e})|\big)$ {\rm (}solid line{\rm )} for the matrix $S$ in Example \ref{exa1} with $l=4$.}
	\label{fig1}
\end{figure}
In Figure \ref{fig1}, we show the five curves of the three relative actual errors
\begin{equation*}
	\log\big(|\tan\angle(q_{2k-1},x_{o})|\big),\ \ \log\big(|\tan\angle(q_{2k},x_{e})|\big), \ \  \log\big(|\rho_{k}-\sigma_{1}|\big)
\end{equation*}
and the error bounds for the former two: 
\begin{equation*}
	\log\big(\big(\frac{\sigma_{2}}{\sigma_{1}}\big)^{2k}|\tan\angle(q_{0},x_{e})|\big)\ \ \mbox{and}\ \  \log\big(\big(\frac{\sigma_{2}}{\sigma_{1}}\big)^{2k-1}|\tan\angle(q_{0},x_{e})|\big)
\end{equation*}
versus iteration number IT. It can be seen from Figure \ref{fig1}(a) that the error of $\rho_k$ is approximately squares of the errors of $q_{2k-1}$ and $q_{2k}$. This verifies the claims in Remark \ref{tanoe}. From Figure \ref{fig1}(b)-(c), one can see that the error bounds (\ref{sino}) and (\ref{sine1}) in Theorem \ref{ConvergenceSTS} are sharp.

\begin{example}\label{exa2}
In this example, we test the skew-symmetric matrices {\sf plsk1919} and {\sf plskz362} 
from the 2D/3D problems in the SuiteSparse Matrix
Collection \cite{Davis2011TheUO}. These two matrices are $1919\times 1919$ and $362\times 362$. 
We compute their complex conjugate dominant eigenpairs with $s=1,5,10,15$, and Table \ref{tab:comp4} displays the results. Figure \ref{fig2} shows the $5$ curves of the relative residual norms of \Cref{alg-MpowerSTS} for solving these two problems when $s=5$.

\end{example}
\begin{table}[h]
	\caption{The IT, TIME, MV's used by \Cref{alg-MpowerSTS} to compute the complex conjugate dominant eigenpairs of $S$ in Example \ref{exa2}.
		\label{tab:comp4}}
	\begin{center}
		\def\arraystretch{1.3}
		\small
		\centering\tabcolsep6.0pt
		\begin{tabular}{|c|cccc|cccc|}
			\hline
			Matrix	&\multicolumn{4}{c|}{$s=1$}&\multicolumn{4}{c|}{$s=5$}\tabularnewline
			& IT & TIME &MV & aveCPU& IT & TIME &MV & aveCPU\tabularnewline	\hline
			plsk1919	&  127& 0.0050& 255& 3.9370e-05&3280&0.2094& 6565&6.3841e-05\tabularnewline	
			plskz362	&  167 &0.0012& 335&7.1856e-06& 2808& 0.0616& 5621&2.1937e-05\tabularnewline	
			\hline
			Matrix	& \multicolumn{4}{c|}{$s=10$}&\multicolumn{4}{c|}{$s=15$}\tabularnewline
			&  IT & TIME &MV & aveCPU& IT & TIME &MV & aveCPU\tabularnewline	\hline
			plsk1919	&  6062& 0.6059&12134& 9.9951e-05&10001&1.4731& 20017&1.4730e-04\tabularnewline	
			plskz362	&5719& 0.1613& 11448&  2.8204e-05&10299& 0.3512& 20613&3.4100e-05
			\tabularnewline	
			\hline
		\end{tabular}
	\end{center}
\end{table}

From Table \ref{tab:comp4}, we see that  
\Cref{alg-MpowerSTS} converges for computing the $15$ complex conjugate dominant eigenpairs of the {\sf 
plsk1919} and {\sf plskz362} matrices. Also, it can be seen from this table that the average CPU time 
aveCPU of \Cref{alg-MpowerSTS} per iteration increases as $s$ does. This is in accordance with the results 
mentioned earlier in the numerical reports on \Cref{exa1}.

\begin{figure}[h!]
	\centering\tabcolsep=-10mm
	{\footnotesize{\doublerulesep1pt
			\begin{tabular}
				[c]{c}%
				\includegraphics[height=3.5in]{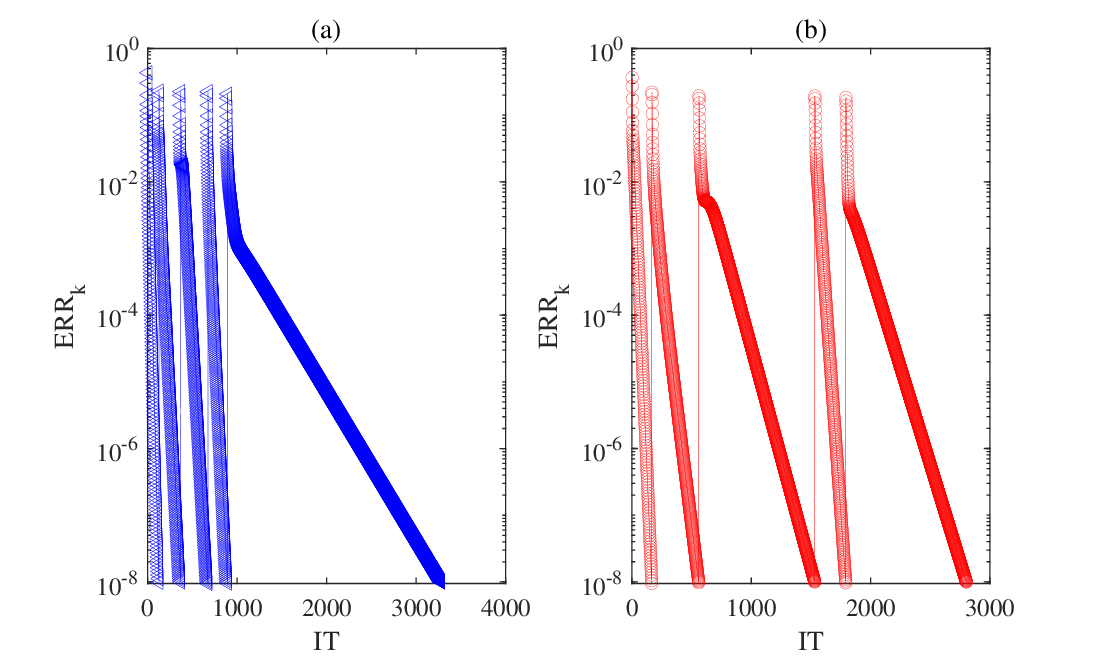}
	\end{tabular}}}
	\caption{(a) and (b): convergence curves of the relative residual norms of \Cref{alg-MpowerSTS} for {\sf plsk1919} and {\sf plskz362}, respectively.}
	\label{fig2}
\end{figure}

\newpage
\begin{example}\label{exa3}
To further show the effectiveness of the proposed method, we report the IT, TIME and MV's computed by \Cref{alg-MpowerSTS} for the matrices $S=(A-A^{T})/{2}$ and $S=[0,A;-A^{T},0]$ with $A$ being the square or rectangular real matrices from the University of Florida Sparse Matrix Collection \cite{Davis2011TheUO}. Table \ref{matrix} shows the properties of the tested matrices $A$, where "nnz($S$)" denotes the number of nonzero entries of $S$ and "description" denotes the problem classification of the corresponding matrix $A$.  \Cref{tab:comp2}--\Cref{tab:comp3} and \Cref{fig4} show the corresponding numerical results. 
\end{example}
\begin{table}[h]
	\caption{Properties of the test matrices $S=(A-A^{T})/2$ (top) and $S=[0,A;-A^{T},0]$ (bottom). 
	\label{matrix}}
	\begin{center}
		\def\arraystretch{1.2}
		\centering\tabcolsep7pt
		\begin{tabular}{lccr}
			\hline
			$A$  &  $n$ & nnz($S$)& description \tabularnewline
			\hline
			epb3  &  84617 &493102& Thermal problem 	\tabularnewline	
			barth &  6691&39496&Duplicate structural problem	\tabularnewline
			utm5940  &  5940&114590&Electromagnetics problem  \tabularnewline
			e40r0100  &  17281& 906340& 2D/3D problem \tabularnewline
			ns3Da  & 20414& 1660392& Fluid dynamics problem\tabularnewline
			Zd$_{-}$Jac6 &  22835& 3422730&Chemical process simulation \tabularnewline
			inlet  &  11730& 440562& Model reduction problem \tabularnewline
			memplus& 17758& 42682&Circuit simulation problem\tabularnewline
			power197k& 197156& 924070&Power Network problem\tabularnewline		
			\hline
			$A$  &  size($A$) & nnz($S$)& description  \tabularnewline
			\hline
			m3plates  & $11107\times 11107$&13278& Acoustics problem	\tabularnewline	
			bcsstm35 &  $30237\times 30237$& 41238& Structural problem\tabularnewline
			c-big  &$345241\times 345241$&4681718&Directed weighted graph problem\tabularnewline
			jan99jac080sc  &$27534\times 27534$&302126	&Economic problem \tabularnewline
			stormg2-8  & $4409\times 11322$&57106& Economic problem \tabularnewline
			M80PI$_{-}$n1&  $4028\times4028$& 8066&Model reduction problem\tabularnewline
			DIMACS10 & $56438\times 56438$&1203204& Random unweighted graph \tabularnewline
			Hardesty2 & $929901\times 	303645$&8041462&Computer graphics problem \tabularnewline
			ML$_{-}$Graph & $10000\times10000 $&291200&Undirected weighted graph  \tabularnewline	
			gupta2 & $62064\times 	62064 $&8496572&Optimization Problem \tabularnewline	
			\hline
		\end{tabular}
	\end{center}
\end{table}

\begin{table}[!h]
	\caption{The IT, TIME, MV's of \Cref{alg-MpowerSTS} to compute the complex conjugate dominant eigenpairs of $S=(A-A^{T})/2$.
		\label{tab:comp2}}
	\begin{center}
		\def\arraystretch{1.2}
		\small
		\centering\tabcolsep7.5pt
		\begin{tabular}{c|ccc|ccc}
			\hline
			\text{Matrix} &
			\multicolumn{3}{c|}{$s=1$} & 	
			\multicolumn{3}{c}{$s=5$}\tabularnewline
			&  IT & TIME &MV & IT & TIME & MV\tabularnewline	\hline
			epb3  & 611& 0.6223& 1223&7860& 37.0501& 15725
			\tabularnewline	
			barth & 2251& 0.3187&4503&50446& 16.3731& 100897
			\tabularnewline
			utm5940  &  315& 0.0580& 631& 13109& 5.4587& 26223
			\tabularnewline
			e40r0100  &  171& 0.2386& 343&5610& 10.3717& 11225
			\tabularnewline
			ns3Da  & 562& 1.9836& 1125&9247& 40.5276& 18499
			\tabularnewline
			inlet  &829&0.5291& 1659&33256& 32.0562& 66517
			\tabularnewline
			Zd$_{-}$Jac6 & 2& 0.0171&5&   386& 2.6058& 777
			\tabularnewline
			memplus&  18&0.0032& 37&    40255&43.0547& 80515\tabularnewline
			power197k & 51& 0.3860&103&886& 16.1148& 1777
			\tabularnewline		
			\hline
		\end{tabular}
	\end{center}
\end{table}

\begin{table}[h]
	\caption{The IT, TIME, MV's of \Cref{alg-MpowerSTS} to compute the complex conjugate dominant eigenpairs of $S=[0,A;-A^{T},0]$.
		\label{tab:comp3}}
	\begin{center}
		\def\arraystretch{1.2}
		\small
		\centering\tabcolsep7.5pt
		\begin{tabular}{c|ccc|ccc}
			\hline
			\text{Matrix} &
			\multicolumn{3}{c|}{$s=1$} & 	
			\multicolumn{3}{c}{$s=5$}\tabularnewline
	&  IT & TIME &MV &   & TIME &MV\tabularnewline	\hline
	m3plates  &  84& 0.0257& 169&3595& 3.6484& 7195
	\tabularnewline	
	bcsstm35 & 48& 0.0256& 97&3452& 13.8056& 6909
	\tabularnewline
	   c-big  & 1279&23.8567& 2559&1937& 70.6949& 3879
	   \tabularnewline
	jan99jac080sc  &  15& 0.0204& 31&47125& 235.3269& 94255
	 \tabularnewline
	stormg2-8  &191& 0.0370&383& 43191& 25.2662& 86387
	\tabularnewline
	M80PI$_{-}$n1& 26& 0.0036 &53&10887& 2.2411&21779\tabularnewline
	 DIMACS10 & 36& 0.1528& 73&1575& 18.7682& 3155\tabularnewline
	Hardesty2 & 192&5.2684&385&7149& 775.1950& 14303\tabularnewline
	ML$_{-}$Graph & 21&0.0169& 43&3896& 6.5166& 7797
	\tabularnewline		
gupta2& 822&11.9290& 1645&4126&75.9087& 8257
	\tabularnewline		
			\hline
		\end{tabular}
	\end{center}
\end{table}

\begin{table}[h]
	\caption{The IT, TIME, MV's and aveCPU of \Cref{alg-MpowerSTS} when the matrix $A$ is taken as {\sf e40r0100}, {\sf memplus}, {\sf gupta2} and {\sf Hardesty2} shown in \Cref{tab:comp2} and \Cref{tab:comp3} for each fixed $s$, $1\leq s\leq5$.
		\label{tab:test7}}
	\begin{center}
		\def\arraystretch{1.2}
		\small
		\centering\tabcolsep7.5pt
		\begin{tabular}{c|c|ccccc}
			\hline
			\text{Matrix}& &$s=1$ & $s=2$&$s=3$& $s=4$&$s=5$\tabularnewline\hline
			&  IT&171  &  818&  3221  &  766 &   612	\tabularnewline	
			e40r0100 &  TIME&0.2386 & 1.3585  &  5.6798 &   1.4216 &   1.6732 \tabularnewline	
			& MV&343&  1637 &  6443  & 1533&  1225	\tabularnewline	
			& aveTIME&0.0014& 0.0017 & 0.0018 & 0.0019 &  0.0027	\tabularnewline	\hline
			&  IT& 18  &  5  &199  & 20000& 20000	\tabularnewline	
			memplus &  TIME& 0.0032&  0.0020&  0.0930& 10.7614 &  32.1951
			\tabularnewline	
			& MV& 37 & 11  & 399& 40001 & 40001	\tabularnewline	
			& aveTIME& 0.0002 & 0.0004 & 0.0005 & 0.0005& 0.0016	\tabularnewline	\hline
			&  IT& 822  &   1714   &   724  &   400  &  466\tabularnewline	
			gupta2 &  TIME&11.9290 & 29.2450 & 14.3635 & 8.8176 &11.5536 \tabularnewline	
			& MV& 1645  &   3429 &  1449    &  801 &     933\tabularnewline	
			& aveTIME& 0.0145& 0.0171 &0.0198  & 0.0220 & 0.0248\tabularnewline	\hline
			&  IT& 192 &   164  &  2216 & 3351 &  1178	\tabularnewline	
			Hardesty2 &  TIME&  5.4053 & 9.8620& 200.3689& 392.4680 &167.0908 \tabularnewline	
			& MV& 385 & 329 &   4433  &  6703 &2357	\tabularnewline	
			& aveTIME& 0.0282&  0.0601 & 0.0904 & 0.1171 &  0.1418\tabularnewline	\hline		

		\end{tabular}
	\end{center}
\end{table}
Tables~\ref{tab:comp2}-\ref{tab:comp3} show that \Cref{alg-MpowerSTS} succeeds for all the matrices in Table \ref{matrix} when $s=1$ and $s=5$. To explore more features of \Cref{alg-MpowerSTS}, we present \Cref{tab:test7} to show the IT, TIME, MV's and aveTIME for computing each complex conjugate dominant eigenpair with $s$ increasing from $1$ to $5$ when the matrices $A$ are {\sf e40r0100}, {\sf memplus}, {\sf gupta2} and {\sf Hardesty2}. From \Cref{tab:test7}, for a fixed $s, \ 1\leq s\leq5$, one can observe 
that the aveCPU of {\sf e40r0100} is more than that of {\sf memplus}. Details can be seen in \Cref{fig4}(a). This is because the two matrices have about the same order, but the proportion of non-zero elements of {\sf e40r0100} is larger than that of {\sf memplus}. In this case, the time for computing the two matrix-vector products $Sq_{2k}$ and $Sq_{2k+1}$ in step 5 of \Cref{alg-MpowerSTS} for {\sf e40r0100} is longer than that of {\sf memplus}. Meanwhile, as we can see from \Cref{tab:test7} and \Cref{fig4}(b) the aveCPU of {\sf gupta2} is less than that of {\sf Hardesty2}. The reason is that the order of {\sf gupta2} is much less than that of {\sf Hardesty2}, but the numbers of non-zero elements in these two matrices are almost the same. At this time, the costs of two summations of the deflation process in step 5 of \Cref{alg-MpowerSTS} for {\sf Hardesty2} will be more expensive than that of the {\sf gupta2}. So the aveCPU of $S$ for a fixed $s$ increases monotonically with respect to the proportion of non-zero elements. If the number of non-zero elements is fixed, then the aveCPU of a fixed $s$ will increase as the order of $S$ does. 

\begin{figure}[h!]
	\centering\tabcolsep=-10mm
	{\footnotesize{\doublerulesep1pt
			\begin{tabular}
				[c]{c}%
				\includegraphics[height=3.2in]{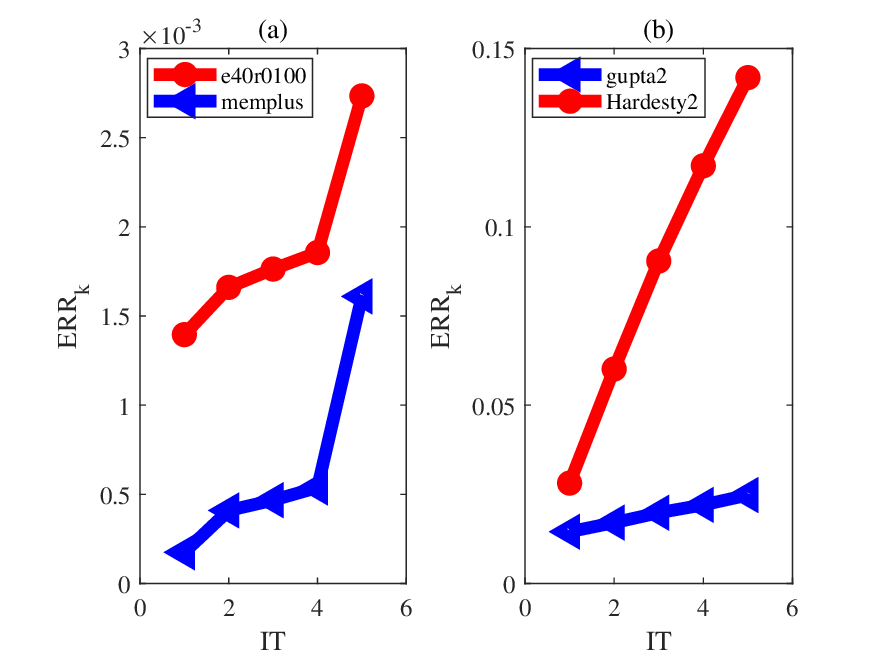}
	\end{tabular}}}
	\caption{The aveCPU for each complex conjugate dominant eigenpair for $s$ from $1$ to $5$.}
	\label{fig4}
\end{figure}

\section{Conclusion}\label{con}
In this paper, we have proposed an SSP method to fix the shortcomings of the classical power method that cannot compute the complex conjugate dominant eigenpairs of the real skew-symmetric matrix $S$. We have presented a quantitative convergence result for a general normal matrix. Based on this, we have established the convergence results of the SSP method, i.e., \Cref{alg-powerSTS} and designed a reliable and general-purpose stopping criterion. Finally, we have proposed and developed a deflation technique to compute several complex conjugate dominant eigenpairs of $S$ by deflating the converged $2i, \ i=1,2,\ldots,s-1$ complex conjugate dominant eigenvalues $\lambda_{\pm j}=\pm\mi\sigma_{j}$ and their corresponding eigenvectors $x_{\pm j}, \ j=1,2,\ldots,i$, with $s$ being the desired number of complex conjugate dominant eigenpairs.

Subspace or simultaneous iteration \cite{Golub,Saad,Stewart} is a block generalization of the power method, and can compute several dominant eigenpairs of a given large sparse matrix. It converges faster than 
the power method. The refined subspace iteration presented by Jia \cite{JiaANM} is an improvement on the classical subspace iteration. Therefore, how to extend the SSP method to a block variant for 
the eigenvalue problem of the skew-symmetric $S$ is certainly interesting and deserves attention. This will be our future work.

\section*{Declarations}

The author declares that she read and approved the final manuscript.


\section*{Acknowledgements}
We thank Professor Zhongxiao Jia of Tsinghua University for his very careful reading of the paper and for his
valuable comments and suggestions, which help us improve and complete the presentation.

\section*{Data Availability}
Inquiries about data availability should be directed to the authors.

\bibliographystyle{siamplain}
\bibliography{references}
\end{document}